\documentclass[reqno]{amsart}

\usepackage{amsmath}
\usepackage{amssymb}
\usepackage{amsthm}
\usepackage{color}
\usepackage{graphicx}
\usepackage{a4wide}
\usepackage{hyperref}
\usepackage{enumitem}

\title{Characteristic polynomials of the weak order on classical and affine
  Coxeter groups}

\author{Jang Soo Kim}
\address{
Department of Mathematics, Sungkyunkwan University, Suwon,
South Korea}
\email{jangsookim@skku.edu}

\author{Sun-mi Yun}
\address{
Department of Mathematics, Sungkyunkwan University, Suwon,
South Korea}
\email{sera314@skku.edu}

\date{\today}


\newtheorem{thm}{Theorem}[section]
\newtheorem{lem}[thm]{Lemma}
\newtheorem{prop}[thm]{Proposition}
\newtheorem{cor}[thm]{Corollary}
\theoremstyle{definition}
\newtheorem{defn}[thm]{Definition}

\newcommand\SG{\mathfrak{S}}

\newcommand\Alt{\operatorname{Alt}}

\newcommand\flr[1]{\left\lfloor #1\right\rfloor}

\newcommand\MD{\mathcal{D}}
\newcommand\wt{\operatorname{wt}}

\begin{document}

\begin{abstract}
  We find a simple product formula for the characteristic polynomial of the
  permutations with a fixed descent set under the weak order. As a corollary we
  obtain a simple product formula for the characteristic polynomial of
  alternating permutations. We generalize these results to Coxeter
  groups. We also find a formula for the generating function for the
  characteristic polynomials of classical Coxeter groups, which is then related
  to affine Coxeter groups.
\end{abstract}

\maketitle

\section{Introduction}

The M\"obius function of a poset is an important topic of study in many areas of
mathematics including number theory, topology, and combinatorics. For example,
if the poset is the set of positive integers ordered by the division relation,
then its M\"obius function is the classical M\"obius function in number theory.
The M\"obius function of a poset expresses the reduced Euler characteristic of
the simplicial complex coming from the poset. The principle of inclusion and
exclusion is also generalized in terms of M\"obius functions. See \cite[Chapter
3]{EC1} for more details on M\"obius functions.

The characteristic polynomial of a finite ranked poset is a generating function
for the M\"obius function on the poset. A useful application of characteristic
polynomials is that one can compute the number of regions and bounded regions of
a hyperplane arrangement using the characteristic polynomial of its intersection
poset, see \cite{Stanley2004}.

There are many interesting posets whose characteristic polynomials factor
nicely. For example, the characteristic polynomials of the Boolean poset and the
partition poset are given by \( (q-1)^n \) and \( (q-1)(q-2)\cdots (q-n+1) \)
respectively. See \cite{Blass1997, Sagan1999, Stanley1972} for more
examples.

The main objective of this paper is to study the characteristic polynomials of
the (left) weak order on classical and affine Coxeter groups. The motivation of
this paper was the following observation on the characteristic polynomial of the
poset $\Alt_n$ of alternating permutations ordered by the weak order. Here an
alternating permutation is a permutation $\pi=\pi_1\pi_2\pi_3\dots\pi_n$
satisfying $\pi_1<\pi_2>\pi_3<\cdots$.

\begin{thm}\label{thm:alt}
  The characteristic polynomial of $\Alt_n$ is
  \[
    \chi_{\Alt_n}(q)=q^{\binom{n-1}{2}-\flr{\frac{n}{2}}}(q-1)^{\flr{\frac{n}{2}}}.
  \]    
\end{thm}

Let $\SG_n$ denote the symmetric group of order \( n \), that is, the set of
permutations on \( [n]:=\{1,2,\dots,n\} \). Note that an element in \( \SG_n \)
is an alternating permutation if and only if its (right) descent set is equal to
\( \{2,4,6,\dots\}\cap [n] \). In Theorem~\ref{thm:fxch} we show that the
characteristic polynomial of the subposet of $\SG_n$ with a fixed descent set
has a simple factorization. Theorem~\ref{thm:alt} then follows immediately from
Theorem~\ref{thm:fxch}.

A Coxeter group is a group defined by generators and certain relations. Coxeter
groups are also studied in many different areas of mathematics, and their
classification is well known. In particular, the symmetric group $\SG_n$ is the
Coxeter group of type $A_{n-1}$, see \cite[Proposition 1.5.4]{BjornerBrenti}.
There are two important orders on Coxeter groups, the Bruhat order and the
(left) weak order. In this paper we will only consider the weak order. See
\cite{Bjorner1984,BjornerBrenti,Jedlivcka2005,Tenner2020} and references therein
for combinatorial properties of the weak order.

In search of a generalization of Theorem~\ref{thm:alt} we study characteristic
polynomials of Coxeter groups under the weak order. By slightly modifying the
definition of the characteristic polynomial so that it can be defined on an
infinite poset, we also compute (modified) characteristic polynomials of affine
Coxeter groups.

The rest of this paper is organized as follows.

In Section~\ref{sec:preliminaries}, we give some definitions and known results
on posets and Coxeter groups. In Section~\ref{sec:char-polyn-an}, we show that
the characteristic polynomial of an interval of a Coxeter group with the weak
order is decomposed into the product of characteristic polynomials of its
subgroups. In Section~\ref{sec:prop-desc-class}, we show that the descent class
of a Coxeter group is an interval. In Section~\ref{sec:perm-with-fixed} we give
a simple product formula for the characteristic polynomial of the set of
permutations with a fixed descent set. In Section~\ref{sec:modif-char-polyn}, we
slightly modify the characteristic polynomial so that it is defined for affine
Coxeter groups. We then compute the generating functions for the modified characteristic
polynomials of the classical Coxeter groups \( A_n, B_n \), and \( D_n \).
Finally, in Section~\ref{sec:modif-char-polyn-1}, we express the modified
characteristic polynomials for affine Coxeter groups \( \tilde{A}_n,
\tilde{B}_n, \tilde{C}_n, \) and \( \tilde{D}_n \) in terms of the corresponding
finite Coxeter groups.

\section{Preliminaries}
\label{sec:preliminaries}

In this section, we give basic definitions and results on posets and Coxeter groups.

\subsection{Posets}

Let $(P,\le)$ be a poset. For \( x,y\in P \) we write \( x<y \) to mean \( x\le
y \) and \( x\ne y \). For $x,y\in P$, if $x< y$ and there is no element $t\in
P$ such that $x<t<y$, then we say that \emph{$y$ covers $x$} and denote by
$x\lessdot y$. The \emph{(closed) interval} $[x,y]$ is the subset $\{t\in P:x\le
t\le y\}$ of $P$. If there is an element $x\in P$ such that $x\le y$
(resp.~$x\ge y$) for all $y\in P$, then \( x \) is called the \emph{bottom}
(resp.~the \emph{top}) of $P$ and is denoted by $\hat{0}$
(resp.~$\hat{1}$). 

Let $C=\{t_1,t_2,\dots,t_k\}$ be a subset of $P$. If any two elements of $C$ are
comparable, then $C$ is called a \emph{chain} of $P$. The \emph{length} of a
chain $C$ is defined to be $|C|$. A chain of \( P \) is called a \emph{maximal
  chain} if it is not contained in any other chain of \( P \). If every maximal
chain of $P$ has the same length, say $m$, then $P$ is called a \emph{ranked
  poset} (or a \emph{graded poset}). In this case, there is a unique \emph{rank
  function} $r_P :P\rightarrow\{0,1,2,\dots,m\}$ such that $r_P(t)=0$ for every
minimal element $t\in P$ and $r_P(y)=r_P(x)+1$ if $x\lessdot y$. We say that
$r_P(x)$ is the \emph{rank} of $x\in P$. The length $m$ of a maximal chain of
$P$ is called the \emph{rank} of $P$ and denoted by $r_P(P)$.

For two elements $x,y$ of $P$, a \emph{lower bound} of $x$ and $y$ is an element
$u$ such that $u\le x$ and $u\le y$. If $u$ is a lower bound of $x$ and $y$
satisfying $v\le u$ for any lower bound $v$ of $x$ and $y$, then $u$ is called
the \emph{meet} of $x$ and $y$ and denoted by $x\wedge y$. Similarly, an
\emph{upper bound} of $x$ and $y$ is an element $u$ such that $x\le u$ and $y\le
u$. If \( u \) is an upper bound of $x$ and $y$ satisfying $u\le v$ for any
upper bound $v$ of $x$ and $y$, then $u$ is called the \emph{join} of $x$ and
$y$ and denoted by $x\vee y$. For a subset $A=\{a_1,a_2,\dots,a_i\}$ of $P$, we
denote by $\bigvee A$ the join of all elements in $A$, i.e., $\bigvee A=a_1\vee
a_2\vee\cdots\vee a_i$.

The \emph{M\"{o}bius function} $\mu_P$ of $P$ is the unique function 
defined on the pairs \( (x,y) \) of elements \( x,y\in P \) with \( x\le y \) satisfying the following condition:
\[
  \sum_{x\le t\le y} \mu_P(x,t)= \delta_{x,y} \qquad \mbox{for all \( x\le y \) in \( P \),}
\]
where \( \delta_{x,y} \) is \( 1 \) if \( x=y \) and \( 0 \) otherwise.

\begin{defn}
  Let $P$ be a finite ranked poset with $\hat{0}$. The \emph{characteristic
    polynomial} $\chi_P(q)$ of $P$ is defined by
  \[
    \chi_P(q)=\sum_{t\in P}\mu_P(\hat{0},t)q^{r_P(P)-r_P(t)}.
  \]
\end{defn}

\subsection{Coxeter groups}

Let \( S \) be a set and let \( m:S\times S\to \{1,2,\dots,\infty\} \) be a
function such that $m(s,s)=1$ and $m(s,s')=m(s',s)\ge 2$ for $s\neq s'$. Let \(
W \) be the group defined by the generators \( S \) and the relations \(
(ss')^{m(s,s')}=e \) for all \( s,s'\in S \), where \( e \) is the identity
element. In this case we say that $W$ is a \emph{Coxeter group}, $S$ is a set of
\emph{Coxeter generators}, and the pair $(W,S)$ is a \emph{Coxeter system}.
Note that, for any Coxeter generators $s,s'$ with $s\neq s'$, we have
\begin{equation}\label{eq:braid}
  s^2=e\mbox{\qquad and\qquad} \overbrace{ss'ss's\cdots}^{m(s,s')}=\overbrace{s'ss'ss'\cdots}^{m(s,s')}.  
\end{equation}
In particular, distinct Coxeter generators $s$ and $s'$ commute if and only if $m(s,s')=2$. 

A \emph{Coxeter graph} (or a \emph{Coxeter diagram}) is a graph such that the nodes are the Coxeter generators, 
nodes $s$ and $s'$ are connected by an edge if $m(s,s')\ge 3$, and the edge has the label $m(s,s')$ if $m(s,s')\ge 4$. 
A Coxeter system $(W,S)$ is \emph{irreducible} if the Coxeter graph of $(W,S)$ is connected. 
The irreducible Coxeter groups have been classified and every reducible Coxeter group 
decomposes uniquely into a product of irreducible Coxeter groups \cite[p.4]{BjornerBrenti}.

From now on, we assume that $(W,S)$ is a Coxeter system.

Since $S$ is a set of generators of $W$, each
element $w\in W$ can be represented as a product of generators, say
$w=s_1s_2\dots s_k$, where $s_i\in S$. Among all such expressions for $w$, the
smallest $k$ is called the \emph{length} of $w$ and denoted by $\ell(w)$. If \(
k=\ell(w) \), then the expression $s_1s_2\dots s_k$ is called a \emph{reduced
  word} for $w$.

Let \(\alpha_{s,s'}=\overbrace{ss'ss's\cdots}^{m(s,s')}\). Then the second
equation in \eqref{eq:braid} is rewritten by \(\alpha_{s,s'}=\alpha_{s',s}\).
The replacement of \(\alpha_{s,s'}\) by \(\alpha_{s',s}\) in a word is called a
\emph{braid-move}. It is known that every two reduced words of \(w\in W\) is
connected via a sequence of braid-moves, see \cite{Tits1969}.

\begin{defn}
  The Coxeter generators $s\in S$ are called \emph{simple reflections}. 
  For \( w\in W \), we define the following sets:
  \begin{align*}
    &D_L(w)=\{s\in S:\ell(sw)<\ell(w)\},\\
    &D_R(w)=\{s\in S:\ell(ws)<\ell(w)\}.
  \end{align*}
  The set $D_L(w)$ (resp.~$D_R(w))$ is called the \emph{left} (resp.~\emph{right})
  \emph{descent set} of \( w \) and its elements are called \emph{left}
  (resp.~\emph{right}) \emph{descents} of \( w \). Note that $D_L(w)=D_R(w^{-1})$.
\end{defn}

\begin{defn}
  For $I\subseteq J\subseteq S$, the set $\MD^J_I=\{w\in W:I\subseteq
  D_R(w)\subseteq J\}$ is called a \emph{(right) descent class}. We also denote
  $\MD_I=\MD^I_I=\{w\in W:D_R(w)=I\}$ and $W^J=\MD^{S\setminus
    J}_\emptyset=\{w\in W:D_R(w)\subseteq S\setminus J\}$. 
  For
  $J\subseteq S$, the subgroup of $W$ generated by $J$ is called a \emph{parabolic
    subgroup} and denoted by $W_J$. 
\end{defn}

Note that $(W_J,J)$ is also a Coxeter system \cite[Proposition
2.4.1]{BjornerBrenti}.

\begin{defn}
  For $u,w\in W$, we denote by $u\le w$ if $w=s_k\dots s_2s_1u$ for some $s_i\in
  S$ with $\ell(s_i\dots s_2s_1u)=\ell(u)+i$ for $0\le i\le k$. This partial
  order $\le$ is called the \emph{(left) weak order}.
\end{defn}

Throughout this paper we consider a Coxeter group \(W\) as a poset using the
weak order.

If $W$ is finite then $W$ is a lattice \cite{Bjorner1984}. In this case
there exists the top element of \( W \), which we denote by $w_0$. For some
$J\subseteq S$, if $W_J$ is finite, then we denote the top element of $W_J$ by
$w_0(J)$. Since the weak order on $W$ is graded, the rank function is defined on
\( W \). By the definition of the weak order, the rank and the length of each
element are the same.

Let $\mu_W$ denote the M\"{o}bius function of $W$. For $u,w\in W$, it is known
that $\mu_W$ only takes the values \( 0,1 \), and \( -1 \).

\begin{prop}\cite[Corollary 3.2.8]{BjornerBrenti}
  \label{prop:mbf}
  The M\"{o}bius
  function $\mu_W(u,w)$ of $W$ is given by
  \[
    \mu_W(u,w)=
    \begin{cases}
      (-1)^{|J|} & \mbox{if $w=w_0(J)u$ for some $J\subseteq S$,}\\
      0 & \mbox{otherwise.}
    \end{cases}
  \]
\end{prop}
If $u=e$ in Proposition~\ref{prop:mbf}, then $\mu_W(e,w)$ is
nonzero if and only if $w=w_0(J)$ for some $J\subseteq S$. Hence the
characteristic polynomial $\chi_W(q)$ of $W$ can be rewritten as
\begin{equation}\label{eq:chrewrt}
  \chi_W(q)=\sum_{w\in W}\mu_W(e,w)q^{\ell(w_0)-\ell(w)}=\sum_{J\subseteq S} (-1)^{|J|}q^{\ell(w_0)-\ell(w_0(J))}.  
\end{equation}
We note that the M\"{o}bius function \(\mu(u,w)\) for the Bruhat order, which is
another important partial order on a Coxeter group, is given by
\[
  \mu(u,w)=(-1)^{\ell(u)+\ell(w)},
\]
see \cite[Theorem]{Verma1971}.

The following proposition is useful in this paper.

\begin{prop}\cite[Proposition 3.1.6]{BjornerBrenti}
  \label{prop:loweritv}
  If \(u\le w\), then \([u,w]\cong[e,wu^{-1}]\).
\end{prop}

\section{The characteristic polynomial of an interval of a Coxeter group}
\label{sec:char-polyn-an}

In this section we show that the characteristic polynomial for 
an interval of a Coxeter group $W$ is decomposed into the product of
the characteristic polynomials of some Coxeter subgroups of $W$.

Recall that since the weak order on a finite Coxeter group is a lattice, we can
define its characteristic polynomial. Even if a Coxeter group is infinite, its
interval can be thought of as a finite graded poset with the bottom and top
elements so the characteristic polynomial of the interval is also defined.

\begin{lem}\label{lem:join}
Let $I, J\subseteq S$ such that $W_I$ and $ W_J$ are finite subgroups of $W$. 
Then we have
\[
w_0(I)\vee w_0(J)=w_0(I\cup J).
\]
\end{lem}

\begin{proof}
Since $W_I$ and $W_J$ are finite, we have $\bigvee I=w_0(I)$ and $\bigvee J=w_0(J)$, 
see \cite[Lemma 3.2.3]{BjornerBrenti}. Hence  
\[
w_0(I)\vee w_0(J)=\left(\bigvee I\right)\vee\left(\bigvee J\right)=\bigvee(I\cup J)=w_0(I\cup J).
\qedhere
\]
\end{proof}

\begin{prop}\label{prop:max}
For $w\in W$, let $\mathcal{M}(w)=\{u\in W:u\le w\mbox{ and } u=w_0(J)\mbox{ for some }J\subseteq S\}$. 
Then, $w_0(D_R(w))$ is the unique maximal element of $\mathcal{M}(w)$.
\end{prop}

\begin{proof}
  Let $K$ be the union of subsets $J\subseteq S$ such that $w_0(J)\le w$, i.e.,
  $K=\bigcup_{w_0(J)\le w} J$. By Lemma~\ref{lem:join} and the definition of
  $\mathcal{M}(w)$, we have
  \begin{equation}\label{eq:w0K}
    w_0(K)=\bigvee\{w_0(J):w_0(J)\le w, J\subseteq S\}=\bigvee \mathcal{M}(w) ,
  \end{equation}
  which means that $w_0(J)\le w_0(K)$ for all $J$ with \( w_0(J)\le w \). By the
  definition of join, we have $w_0(K)=\bigvee_{w_0(J)\le w} w_0(J)\le w$, which
  implies $w_0(K)\in\mathcal{M}(w)$. By \eqref{eq:w0K}, this implies that
  $w_0(K)$ is the unique maximal element of $\mathcal{M}(w)$.

  Now it remains to show that \( K=D_R(w) \). 
  Suppose $s\in K\subseteq S$. 
  Since $s\le w_0(K)\le w$ and by the definition of the weak order, 
  there is a reduced expression of $w$ ending with $s$. 
  This means that $s\in D_R(w)$.
  Conversely, suppose $s\in D_R(w)$. 
  Then, since $s\le w$ and $s=w_0(\{s\})\in\mathcal{M}(w)$, we have $s\in K$. 
  Hence we obtain $K=D_R(w)$, which completes the proof.
\end{proof}

To decompose the characteristic polynomial of an interval of $W$, we first
decompose the interval itself into a product of Coxeter subgroups of $W$ as
follows.

\begin{lem}\label{lem:intdecom}
  Let $K$ be a subset of $S$ such that $W_K$ is finite. 
  Let $K_1,K_2,\dots,K_r$ be the subsets of $K$ such that each $K_i$ is a connected component 
  of the Coxeter graph for $(W_K,K)$ and $\bigcup^r_{i=1} K_i=K$. 
  Then the interval $[e,w_0(K)]$ of $W$ is isomorphic to 
  $W_{K_1}\times W_{K_2}\times\cdots\times W_{K_r}$ as posets.
\end{lem}

\begin{proof}
  It is well known \cite[Proposition 6.1]{Humphreys1990} that \( W_K \) is
  isomorphic to $W_{K_1}\times W_{K_2}\times\cdots\times W_{K_r}$ as Coxeter
  groups, hence as posets. Thus it suffices to show that $[e,w_0(K)]=W_K$. For
  $u\in W$, let $u\in[e,w_0(K)]$. 
  Since $u\le w_0(K)$, there is a reduced expression \( s_{i_k}\dots s_{i_1} \) of $w_0(K)$ such that \( s_{i_j}\dots s_{i_1}=u \),
  where \( j\le k \). Since every reduced word of $w_0(K)\in W_K$ 
  consists of elements in $K$, we have $u\in W_K$. Conversely, if
  $u\in W_K$ then $u\le w_0(K)$ since $w_0(K)$ is the unique maximal element in
  $W_K$. Thus $u\in[e,w_0(K)]$, which completes the proof.
\end{proof}

The following lemma shows that the characteristic polynomial of the direct
product of posets is the product of the characteristic polynomials of the
posets.

\begin{lem}\label{lem:prod}
  Let \( P \) and \( Q \) be finite graded posets with the bottom and top
  elements. Then
\[
  \chi_{P\times Q}(q)=\chi_P(q)\chi_Q(q).
\]
\end{lem}

\begin{proof}
  It is well known \cite[Proposition 3.8.2]{EC1} that $\mu_{P\times
    Q}((s,t),(s',t'))=\mu_P(s,s')\mu_Q(t,t')$ for all $s,s'\in P$ and $t,t'\in
  Q$. It is easy to see that $r_{P\times Q}((s,t))=r_P(s)+r_Q(t)$. Hence,
\begin{align*}
\chi_{P\times Q}(q)&=\sum_{(s,t)\in P\times Q} \mu_{P\times Q}(\hat{0}_{P\times Q},(s,t))q^{r_{P\times Q}(P\times Q)-r_{P\times Q}((s,t))}\\
&=\sum_{s\in P,\ t\in Q} \mu_P(\hat{0}_P,s)\mu_Q(\hat{0}_Q,t)q^{(r_P(P)+r_Q(Q))-(r_P(s)+r_Q(t))}\\
&=\chi_P(q)\chi_Q(q). \qedhere
\end{align*}
\end{proof}

Let $u,w\in W$ with $u\le w$. Then the interval $[u,w]$ is 
a graded poset with the bottom element $u$ and the top element $w$. 
To compute the characteristic polynomial of $[u,w]$, 
we need the rank function of the poset $[u,w]$. 
Since the rank of $[u,w]$ is the length of maximal chains of $[u,w]$, 
it is the same as the difference of the length of $w$ and the length of $u$, i.e., $\ell(w)-\ell(u)$. 
Similarly, the rank of an element $t\in [u,w]$ is $\ell(t)-\ell(u)$. 

For simplicity, we use the notation $\chi_K$ instead of $\chi_{W_K}$ for the
characteristic polynomial of \(W_K\) for $K\subseteq S$.

\begin{thm}\label{thm:chdecom}
  Let $u,w\in W$ with $u\le w$ and let $K=D_R(wu^{-1})$. 
  Suppose that $K_1,K_2,\dots,K_r$ are the subsets of $K$ such that each $K_i$ is a connected component of the Coxeter graph for $(W_K,K)$ and $\bigcup^r_{i=1} K_i=K$. 
  Then, the characteristic polynomial of $[u,w]$ is given by
\[
\chi_{[u,w]}(q)=q^{\ell(wu^{-1})-\ell(w_0(K))}\chi_{K_1}(q)\chi_{K_2}(q)\cdots\chi_{K_r}(q).
\]
\end{thm}

\begin{proof}
  By Proposition~\ref{prop:loweritv}, we have
\begin{equation}\label{eq:1}
\chi_{[u,w]}(q)=\chi_{[e,wu^{-1}]}(q)=\sum_{t\in [e,wu^{-1}]} \mu_{[e,wu^{-1}]}(e,t)q^{\ell(wu^{-1})-\ell(t)}.
\end{equation}
By Proposition~\ref{prop:mbf}, we have $\mu_{[e,wu^{-1}]}(e,t)\ne 0$ if and only if $t=w_0(J)$ for some
$J\subseteq S$. Therefore we can rewrite \eqref{eq:1} as
\begin{equation}
  \label{eq:2}
  \chi_{[u,w]}(q)=\sum_{t\in \mathcal{M}({wu^{-1}})} \mu_{[e,wu^{-1}]}(e,t)q^{\ell(wu^{-1})-\ell(t)},
\end{equation}
where
$\mathcal{M}({wu^{-1}})=\{v\in W:v\le wu^{-1}\mbox{ and } v=w_0(J)\mbox{ for some }J\subseteq S\}$.

Let $K=D_R(wu^{-1})$. By Proposition~\ref{prop:max}, we have
$\mathcal{M}({wu^{-1}})\subseteq [e,w_0(K)]$ and \( w_0(K)\le wu^{-1} \). By
Proposition~\ref{prop:mbf} again, if \( t\in [e,w_0(K)]\setminus \mathcal{M}({wu^{-1}}) \),
then \( \mu_{[e,wu^{-1}]}(e,t)= 0 \). Therefore we can rewrite \eqref{eq:1} as
\begin{equation}
  \label{eq:3}
  \chi_{[u,w]}(q)=\sum_{t\in [e,w_0(K)]} \mu_{[e,wu^{-1}]}(e,t)q^{\ell(wu^{-1})-\ell(t)}.
\end{equation}
Since \( w_0(K)\le wu^{-1} \), we have \( \mu_{[e,wu^{-1}]}(e,t) =
\mu_{[e,w_0(K)]}(e,t) \) for all \( t\in [e,w_0(K)] \).
Therefore the right-hand side of \eqref{eq:3} is equal to
\[
q^{\ell(wu^{-1})-\ell(w_0(K))}\sum_{t\in [e,w_0(K)]} \mu_{[e,w_0(K)]}(e,t)q^{\ell(w_0(K))-\ell(t)}
  =q^{\ell(wu^{-1})-\ell(w_0(K))}\chi_{[e,w_0(K)]}(q).
\]
Finally, by Lemmas~\ref{lem:intdecom} and \ref{lem:prod}, we have
\[
  \chi_{[e,w_0(K)]}(q) = \chi_{K_1}(q)\chi_{K_2}(q)\cdots\chi_{K_r}(q), 
\]
which completes the proof.
\end{proof}

\section{Properties of descent classes}
\label{sec:prop-desc-class}

In this section we show that the descent class $\MD^J_I$ of a Coxeter group $W$
is an interval of $W$ and obtain some properties of the descent classes.

Throughout this section we assume that $(W,S)$ is a finite Coxeter system. Then,
for any $J\subseteq S$, the subgroup $W_J$ is also finite and $w_0(J)$ exists.
Now we show that $\MD^J_I$ is an interval of $W$.

\begin{prop}\label{prop:interval}
For $I\subseteq J\subseteq S$, each descent class $\MD^J_I$ is equal 
to the interval $[w_0(I),w_0w_0(J^c)]$ of $W$, where \( J^c = S\setminus J \).
\end{prop}

\begin{proof}
  Bj{\"o}rner and Wachs \cite{Bjorner1988} showed that \(\MD^J_I=[w_0(I),w_0^{J^c}]\), where \(w_0^{J^c}\) is the top element of \(W^{J^c}\).
By the relation $w_0=w_0^Jw_0(J)$ in \cite[p.44]{BjornerBrenti}, we are done.
\end{proof}

Note that Theorem~\ref{thm:chdecom} implies that the characteristic polynomial
of an interval $[u,w]$ is determined by $D_R(wu^{-1})$. Therefore the
characteristic polynomial of the descent class $\MD^J_I=[w_0(I),w_0w_0(J^c)]$ is
determined by $D_R(w_0w_0(J^c)w_0(I))$. (Note that $(w_0(I))^{-1}=w_0(I)$.) We
will find a simple description for $D_R(w_0w_0(J^c)w_0(I))$. To do this, we need
the following lemma.

\begin{lem}\label{lem:ndes}
  Let $u=u_t\dots u_1$ and $v=v_r\dots v_1$ be reduced expressions of $u,v\in
  W$ such that \( u_i\ne v_j \) for all \( 1\le i\le t \) and \( 1\le j\le r \).
  Then, for each $v_k\in D_R(v)$, there exists $u_i$ such that $v_ku_i\neq u_iv_k$ if and
  only if $v_k\notin D_R(vu)$.
\end{lem}

\begin{proof}
  Let $v_k\in D_R(v)$. Then there is a reduced expression \( v'_r \dots v'_1 \)
  of \( v \) ending with \( v'_1=v_k \).

  For the ``if'' part, suppose that $v_k$ commutes with every $u_i$. Then \(v'_r\dots v'_2
  u_t\dots u_1 v'_1\) is a reduced expression of \( vu \) ending with \( v_k
  \), which implies $v_k\in D_R(vu)$.

  For the ``only if'' part, suppose that there exists $u_i$ that do not commute
  with $v_k$. For a contradiction suppose that \(v_k\in D_R(vu)\). Then there is
  a reduced expression \( w_{r+t}\dots w_1 \) of \(vu\) ending with \(w_1=v_k\).
  Since any two reduced expressions can be obtained from each other via a
  sequence of braid-moves, we can find a sequence \( R_1,\dots, R_p \) of
  reduced expressions of \( vu \) such that \( R_1=v_r\dots v_1u_t\dots u_1
  \), \( R_p=w_{r+t}\dots w_1 \), and each \( R_{j+1} \) is obtained from \( R_j
  \) by applying a single braid-move.

  Let \( s=v_k \) and \( s'=u_i \). Observe that every simple reflection \( s \)
  is to the left of every simple reflection \( s' \) in \( R_1 \). 
  Therefore we can find the smallest integer \( d \) such
  that \( s \) is to the left (resp.~right) of \( s' \) in \( R_d \) (resp.~\(
  R_{d+1} \)). This can only happen if the braid-move
  $\overbrace{ss'ss'\cdots}^{m}=\overbrace{s'ss's\cdots}^{m}$, for some integer
  $m\ge 3$, is used when we obtain \( R_{d+1} \) from \( R_d \). However, since
  \( m\ge3 \), we can apply this braid-move only if there is already at least
  one \( s \) to the right of \( s' \), which is a contradiction because every \(
  s \) is to the left of each \( s' \) in \( R_d \).
  Therefore every \( u_i \) must commute with \( v_k \), which completes the proof.
\end{proof}

Now we give a simple expression for the set $D_R(w_0w_0(J^c)w_0(I))$.

\begin{lem}\label{lem:iplus}
Let $I\subseteq J\subseteq S$ and $I^+=\{v\in S:iv\neq vi\mbox{ for some }i\in I\}$. 
Then we have $D_R(w_0w_0(J^c)w_0(I))=(J\cup I^+)\setminus I$.
\end{lem}

\begin{proof}
  The set \(D_R(w_0w_0(J^c)w_0(I))\) is a subset of $[e,w_0w_0(J^c)w_0(I)]$. By
  Proposition~\ref{prop:loweritv}, the interval \([e,w_0w_0(J^c)w_0(I)]\) is
  isomorphic to \(\MD^J_I=[w_0(I),w_0w_0(J^c)]\) via $x\mapsto xw_0(I)$ for
  \(x\in [e,w_0w_0(J^c)w_0(I)]\). Since \(D_R(w_0w_0(J^c)w_0(I))=\{v\in S: e\le
  v\le w_0w_0(J^c)w_0(I)\}\), by the isomorphism \(D_R(w_0w_0(J^c)w_0(I))\) can
  also be written as \(\{v\in S:w_0(I)\le vw_0(I)\le w_0w_0(J^c)\}\), say \(K\).
  We claim that $K=(J\cup I^+)\setminus I$.

  Let $v\in K$. Since $w_0(I)\le vw_0(I)$, we have $v\notin I$. Since
  $vw_0(I)\in\MD^J_I$, we have $I\subseteq D_R(vw_0(I))\subseteq J$. Note that
  $D_R(w_0(I))=I$ by the maximality of $w_0(I)$. So $D_R(vw_0(I))$ is $I$ or
  $I\cup \{v\}$. In the case of $D_R(vw_0(I))=I$, in other words $v\notin
  D_R(vw_0(I))$, there exists $i\in I$ such that $iv\neq vi$ by
  Lemma~\ref{lem:ndes}. Hence, in this case, we have $v\in I^+$. For the other
  case $D_R(vw_0(I))=I\cup \{v\}$, we have $v\in J$ since $I\cup
  \{v\}=D_R(vw_0(I))\subseteq J$. So we have $v\in (J\cup I^+)\setminus I$.

Conversely, let $t\in (J\cup I^+)\setminus I$ and consider $tw_0(I)$. 
Since $t\notin I$, we have $w_0(I)\le tw_0(I)$. If $t\in J$, then 
$D_R(tw_0(I))$ is $I$ or $I\cup \{t\}$ so $I\subseteq D_R(tw_0(I))\subseteq J$. 
If $t\in I^+$, then $t\notin D_R(tw_0(I))$ by Lemma~\ref{lem:ndes}. 
In other words $D_R(tw_0(I))=I$. 
Hence we have $tw_0(I)\in\MD^J_I$, which means that $t\in K$.
\end{proof}

Lemma~\ref{lem:iplus} will be used in later to compute the characteristic
polynomial of $\MD_I$ for the Coxeter group $A_n$.

\section{Permutations with a fixed descent set}
\label{sec:perm-with-fixed}

In this section we give an explicit formula for the characteristic polynomial of
the set of permutations with a fixed descent set ordered by the weak order. As a
corollary we obtain a simple formula for the characteristic polynomial of the
poset of alternating permutations.

An important example of Coxeter groups is the finite irreducible Coxeter group
$A_{n-1}$, which can be identified with the symmetric group as follows. For a
positive integer $n$, let $[n]:=\{1,2,\dots,n\}$. The symmetric group $\SG_n$ is
the set of all bijections from $[n]$ to $[n]$. Each element $\pi\in\SG_n$ is
called a \emph{permutation} and we write $\pi=\pi_1\pi_2\dots\pi_n$ where
$\pi_i=\pi (i)$. The \emph{simple transposition} $s_i$ is the permutation that
exchanges the integers $i$ and $i+1$ and fixes all the other integers. Let $S$
be the set of simple transpositions. Then $(\SG_n,S)$ is the Coxeter group
$A_{n-1}$, see \cite[Proposition 1.5.4]{BjornerBrenti}.
From now on we identify each Coxeter generator \(s_i\) and its index \(i\).

\begin{figure}
  \centering
  \includegraphics[scale=0.8]{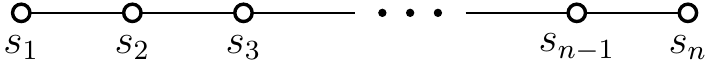}
  \caption{The Coxeter graph of $A_n$.}
  \label{fig:coxgraph}
\end{figure}

The Coxeter graph of \( A_n \) is shown in Figure~\ref{fig:coxgraph}. Note
that a subgraph of the Coxeter graph of \( A_n \) is connected if and only
if the vertices of the subgraph are consecutive integers. Note also that
$s_is_j\neq s_js_i$ if and only if $|i-j|=1$, i.e., $s_i$ and $s_j$ are not
commutative if and only if they are adjacent.

For the Coxeter group $A_n$, we have an explicit formula for the characteristic
polynomial of $\MD_I$ for $I\subseteq S$. 

\begin{thm}\label{thm:fxch}
  For $(A_n,S)$ and $I\subseteq S$, let $d$ be the rank of $\MD_I$. Suppose that
  $S\setminus I=M_1\cup M_2\cup\cdots\cup M_k$ where each $M_i$ is a maximal set
  of consecutive integers. Let $\alpha$, $\beta$ and $\gamma$ be the numbers
  of $i$'s such that $|M_i|=1$, $|M_i|=2$ and $|M_i|\ge3$, respectively. Then,
  the characteristic polynomial of $\MD_I$ is
\[
\chi_{\MD_I}(q)=q^{d-\alpha-2\gamma-3\beta}(q-1)^{\alpha+2\gamma+\beta}(q^2-q-1)^{\beta}.
\]
\end{thm}

\begin{proof}
  By Proposition~\ref{prop:interval}, we have $\MD_I=[w_0(I),w_0w_0(I^c)]$. In
  order to use Theorem~\ref{thm:chdecom}, we investigate the structure of
  $D_R(w_0w_0(I^c)w_0(I))$. By Lemma~\ref{lem:iplus}, we have
  \[
    D_R(w_0w_0(I^c)w_0(I))=(I\cup I^+)\setminus I=I^+\setminus I,
  \]
  where $I^+=\{j\in S:ij\neq ji\mbox{ for some }i\in I\}$. Since
  $I^+=\{j\in S:|i-j|=1\mbox{ for some }i\in I\}$, we have 
\begin{equation}\label{eq:I+}
  I^+\setminus I=\{j\in S\setminus I:|i-j|=1\mbox{ for some }i\in I\}. 
\end{equation}

Let $S\setminus I=M_1\cup M_2\cup\cdots\cup M_k$ as in the statement. By
\eqref{eq:I+}, if $|M_i|=1$ or $|M_i|=2$, every element of $M_i$ is contained in
$I^+\setminus I$, and if $|M_i|\ge3$, only the smallest and largest elements of
$M_i$ are contained in $I^+\setminus I$. Hence every maximal subset of
consecutive integers in $I^+\setminus I$ has one or two elements and the
number of subsets of size \( 1 \) (resp.~\( 2 \)) is $\alpha+2\gamma$
(resp.~$\beta$). Therefore, by Lemma~\ref{lem:intdecom}, we have
\begin{equation}\label{eq:A1A2}
[e,w_0(I^+\setminus I)]\cong (\alpha+2\gamma)A_1\times \beta A_2.
\end{equation} 

Since the rank of $A_1$ is 1 and the rank of $A_2$ is 3, 
the rank of $[e,w_0(I^+\setminus I)]$ is $\alpha+2\gamma+3\beta$. 
Thus, by Theorem~\ref{thm:chdecom} and \eqref{eq:A1A2}, 
\[
  \chi_{\MD_I}(q)=q^{d-(\alpha+2\gamma+3\beta)}\chi_{A_1}(q)^{\alpha+2\gamma}\chi_{A_2}(q)^\beta.
\]
Since $\chi_{A_1}(q)=q-1$ and $\chi_{A_2}(q)=(q-1)(q^2-q-1)$, we obtain the
theorem.
\end{proof}

A permutation $\pi=\pi_1\pi_2\dots\pi_n\in\SG_n$ satisfying
$\pi_1<\pi_2>\pi_3<\cdots$ is called an \emph{alternating permutation}. Let
$\Alt_n$ denote the set of alternating permutations in \( \SG_n \). Then one can
easily check that $\Alt_n=\MD_I$ in the Coxeter system \( (A_{n-1},S) \), where
\( I=\{2,4,6,\dots\}\cap [n-1] \). Observe that \( S\setminus I
=\{1,3,5,\dots\}\cap [n-1] \) in which every maximal subset of consecutive
integers has size \( 1 \). Since $|S\setminus I|=\flr{\frac{n}{2}}$ and the rank
of $\Alt_n$ is $\binom{n-1}{2}$, we obtain the characteristic polynomial of
$\Alt_n$ as a corollary of Theorem~\ref{thm:fxch}.

\begin{cor}
  The characteristic polynomial of $\Alt_n$ is
  \[
  \chi_{\Alt_n}(q)=q^{\binom{n-1}{2}-\flr{\frac{n}{2}}}(q-1)^{\flr{\frac{n}{2}}}.
  \]    
\end{cor}

Stanley \cite{Stanley1972} showed that the characteristic polynomial of a
supersolvable lattice is decomposed into linear factors, see also \cite[Theorem
6.2]{Sagan1999}. Although the characteristic polynomial of \( \Alt_n \) has only
linear factors, one can check that \( \Alt_n \) is not supersolvable.

\section{Modified characteristic polynomials for classical Coxeter groups}
\label{sec:modif-char-polyn}

In this section we find generating functions for modified characteristic
polynomials for the classical Coxeter groups \(A_n\), \(B_n\) and \(D_n\).

\begin{defn}\label{defn:chihat}
Let $W$ be a ranked poset with the bottom element $\hat{0}$. The \emph{modified
  characteristic polynomial} \( \hat{\chi}_W(q) \) of \( W \) is defined by 
\[
\hat{\chi}_W(q)=\sum_{w\in W} \mu_W(\hat{0},w) q^{r_W(w)},
\]
where \(r_W(w)\) is the rank of \(w\in W\).
\end{defn}

If \( W \) is a finite Coxeter group, there is a simple relation between the
modified characteristic polynomial $\hat{\chi}_W(q)$ and the characteristic
polynomial $\chi_W(q)$ as follows.

\begin{prop}\label{prop:chmodch}
Let $(W,S)$ be a finite Coxeter group and $w_0$ the unique maximal element of $W$. Then we have
\[
\hat{\chi}_W(q)=q^{\ell(w_0)}\chi_W(q^{-1}).
\]
\end{prop}

\begin{proof}
 This follows immediately from the definitions of \( \hat{\chi}_W(q)
    \) and $\chi_W(q)$.
\end{proof}

In what follows we find an expression for the generating functions for
$\hat{\chi}_{A_n}(q)$, \(\hat{\chi}_{B_n}(q)\), and \(\hat{\chi}_{D_n}(q)\).
First we consider refinements of these generating functions.

\begin{defn}
  For the Coxeter system $(A_n,S)$ (resp.~$(B_n,S)$, $(D_n,S)$), let $A_{n,k,l}$
  (resp.~$B_{n,k,l}$, $D_{n,k,l}$) be the number of subsets $J\subseteq S$ such that
  $|J|=k$ and $\ell(w_0(J))=l$. We define
  \begin{align*}
    T_A(x,y,z) = \sum_{n\ge0}\sum_{k\ge0}\sum_{l\ge0} A_{n,k,l} x^n y^k z^l,\\
    T_B(x,y,z) = \sum_{n\ge0}\sum_{k\ge0}\sum_{l\ge0} B_{n,k,l} x^n y^k z^l,\\
    T_D(x,y,z) = \sum_{n\ge2}\sum_{k\ge0}\sum_{l\ge0} D_{n,k,l} x^n y^k z^l,
  \end{align*}
  where $A_{0,k,l}=B_{0,k,l}=1$ if $k=l=0$ and $A_{0,k,l}=B_{0,k,l}=0$
  otherwise.
\end{defn}

In the definition of $T_D(x,y,z)$ the sum is over \( n\ge2 \) for computational
convenience. Note that $B_1\cong A_1$, $D_2\cong A_1\times A_1$, and
$D_3\cong A_3$. 

Similar to \eqref{eq:chrewrt}, the polynomial \(\hat{\chi}_W(q)\) can be written
as
\begin{equation}\label{eq:chrewrt1}
  \hat{\chi}_W(q)=\sum_{J\subseteq S} (-1)^{|J|}q^{\ell(w_0(J))}. 
\end{equation}
If \(W=A_n\), we have
\[
  \hat{\chi}_{A_n}(q)=\sum_{k\ge0}\sum_{l\ge0} A_{n,k,l} (-1)^k q^l,
\]
which implies that
  \begin{equation}\label{eq:TA}
    T_A(x,-1,q)=\sum_{n\ge 0} \left(\sum_{k\ge0}\sum_{l\ge0} A_{n,k,l} (-1)^k q^l\right) x^n=\sum_{n\ge 0} \hat{\chi}_{A_n}(q) x^n.
  \end{equation}
 By the same arguments we have
 \begin{align}
   \label{eq:TB}
   T_B(x,-1,q)&=\sum_{n\ge 0} \hat{\chi}_{B_n}(q) x^n,\\
   \label{eq:TD}
   T_D(x,-1,q)&=\sum_{n\ge 2} \hat{\chi}_{D_n}(q) x^n.
 \end{align}

 We use the fact that $\ell(w_0(A_n))=\binom{n+1}{2}$, $\ell(w_0(B_n))=n^2$,
and $\ell(w_0(D_n))=n(n-1)$, see~\cite[p.16, p.80 Table 2]{Humphreys1990}. Here the
 notation $w_0(A_n)$ (resp.~$w_0(B_n)$, \(w_0(D_n)\)) means that the unique maximal
 element of $A_n$ (resp.~$B_n$, \(D_n\)).

 Now we give explicit formulas for $T_A(x,y,z)$, $T_B(x,y,z)$, and $T_D(x,y,z)$.

\begin{prop}\label{prop:genfc}
We have
\begin{align}
  \label{eq:T_A}
  & T_A(x,y,z)=\frac{P}{1-xP},\\
  \label{eq:T_B}
  & T_B(x,y,z)=\frac{Q}{1-xP},\\
  \label{eq:T_D}
  & T_D(x,y,z)=\frac{x^2P+2x(P-1)+R}{1-xP},
\end{align}
where
\[
P=\sum_{n\ge0} x^n y^n z^{\binom{n+1}2},\ Q=\sum_{n\ge0} x^n y^n z^{n^2},\mbox{ and } R=\sum_{n\ge2} x^n y^n z^{n^2-n}.
\]
\end{prop}

\begin{proof}
  Recall that \( A_{n,k,l} \) is the number of subsets $J$ of $S=[n]$ satisfying
  $|J|=k$ and $\ell(w_0(J))=l$ in the Coxeter group \( W=A_n \). For such a
  subset \( J \), let $J=J_1\cup \cdots\cup J_m$, where each $J_i$ is a
  connected component of the Coxeter graph of $(W_J,J)$. Then
  $\ell(w_0(J))=\sum_{i=1}^m\ell(w_0(J_i)) =\sum_{i=1}^m \binom{|J_i|+1}{2}$ and
  the contribution of \( J\subseteq[n] \) to \( T_A(x,y,z) \) is
\[
  \wt(n,J):=  x^ny^{|J_1|+\dots+|J_m|}z^{\binom{|J_1|+1}{2} + \cdots + \binom{|J_m|+1}{2}}.
\]
Therefore 
\begin{equation}\label{eq:nJ}
  T_A(x,y,z) = \sum_{(n,J)\in X} \wt(n,J),
\end{equation}
where \( X \) is the set of all pairs \( (n,J) \) of a nonnegative integer \( n
\) and a set \( J\subseteq[n] \).

We will divide the sum in \eqref{eq:nJ} into two cases \( J=[n] \) and \( J\ne[n] \).
First, note that
\begin{equation}\label{eq:J=[n]}
  \sum_{(n,[n])\in X} \wt(n,[n]) = \sum_{n\ge0} x^n y^n z^{\binom{n+1}2} = P.
\end{equation}
Now consider \( (n,J)\in X \) with \( J\ne[n] \). Let \( t \) be the smallest
positive integer not contained in \( J \). Then we have $J=J_1\cup \cdots\cup
J_m$ with \( J_1=[t-1] \) and \( J\setminus J_1 \subseteq \{t+1,t+2,\dots,n\}
\). Let \( J'=\{j-t:j\in J\setminus J_1\} \). Then \( (n-t,J')\in X \) and \(
\wt(n,J)=x^{t}y^{t-1}z^{\binom{t}{2}}\wt(n-t,J') \). Since every \( (n,J)\in X
\) with \( J\ne[n] \) is obtained from \( (n-t,J')\in X \) for some \( t\ge1 \)
and \( J' \) in this way, we have
\begin{equation}\label{eq:J!=[n]}
  \sum_{(n,J)\in X, J\ne[n]} \wt(n,J) = 
  \sum_{t\ge1}x^{t}y^{t-1}z^{\binom{t}{2}}\sum_{(n-t,J')\in X}\wt(n-t,J')
  = xP\cdot T_A(x,y,z).
\end{equation}
By \eqref{eq:nJ}, \eqref{eq:J=[n]}, and \eqref{eq:J!=[n]}, we have
$T_A(x,y,z)=P+xP\cdot T_A(x,y,z)$, which implies the first identity \eqref{eq:T_A}.

The second identity \eqref{eq:T_B} can be proved similarly if we consider the
Coxeter system \( (B_n,S) \) with \( S=[n] \) and the Coxeter graph as shown
in Figure~\ref{fig:BD}. 
The only difference is that if \( J=[n] \), then \(
\wt(n,J)=x^ny^nz^{n^2} \) since $\ell(w_0(B_n))=n^2$.

\begin{figure}
  \centering
  \includegraphics[scale=0.8]{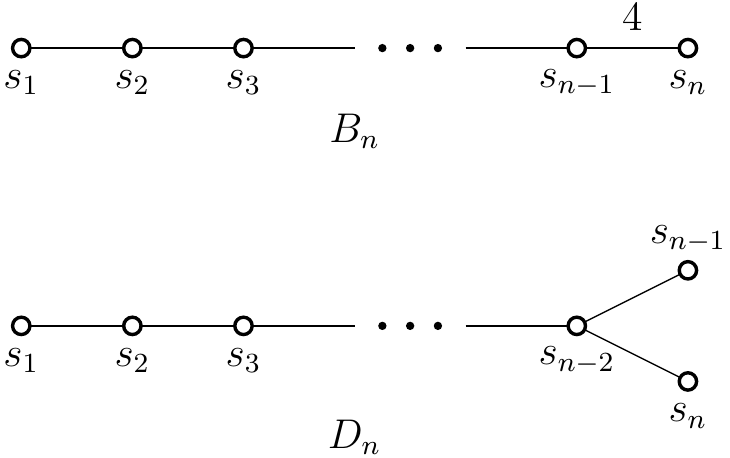}
  \caption{The Coxeter graphs of \( B_n \) and \( D_n \).}
  \label{fig:BD}
\end{figure}

For the third identity \eqref{eq:T_D}, we consider Coxeter system \( (D_n,S) \)
with \( S=[n] \) and the Coxeter graph as shown in Figure~\ref{fig:BD}. The
proof is similar to the case of the first identity except that
$\ell(w_0(D_n))=n^2-n$ and we need to consider the cases that the smallest
positive integer \( t \) not contained in \( J \) is \( n-1 \) or \( n \)
separately. It is not hard to check that \( T_D(x,y,z) \) satisfies
$T_D(x,y,z)=R+xPT_D(x,y,z)+2x(P-1)+x^2P$, which implies the third identity.
\end{proof}

By \eqref{eq:TA}, \eqref{eq:TB}, and \eqref{eq:TD}, substituting $y=-1$ and
$z=q$ in the formulas of Proposition~\ref{prop:genfc} gives the following
formulas for the generating functions of \(\hat{\chi}_{A_n}(q)\),
\(\hat{\chi}_{B_n}(q)\), and \(\hat{\chi}_{D_n}(q)\).

\begin{thm}
We have
\begin{align*}
  & \sum_{n\ge0} \hat{\chi}_{A_n}(q) x^n=\frac{\sum_{n\ge0} (-1)^n q^{\binom{n+1}2} x^n}{\sum_{n\ge0} (-1)^n q^{\binom{n}2} x^n},\\
  & \sum_{n\ge0} \hat{\chi}_{B_n}(q) x^n=\frac{\sum_{n\ge0} (-1)^n q^{n^2} x^n}{\sum_{n\ge0} (-1)^n q^{\binom{n}2} x^n},\\
  & \sum_{n\ge2} \hat{\chi}_{D_n}(q) x^n=\frac{\sum_{n\ge2} (-1)^n (q^{\binom{n-1}2}-2q^{\binom{n}2}+q^{2\binom{n}2}) x^n}{\sum_{n\ge0} (-1)^n q^{\binom{n}2} x^n}.  
\end{align*}
\end{thm}

\section{Modified characteristic polynomials for affine Coxeter groups}
\label{sec:modif-char-polyn-1}

In this section we express the modified characteristic polynomial \(
\hat{\chi}_W(q) \) of an affine Coxeter group \( W \) using those of finite Coxeter
groups.

Recall that the modified characteristic polynomial $\hat{\chi}_W(q)$ is defined
as a series if $W$ is an infinite poset. However, if \( W \) is an affine
Coxeter group with finite generators, \eqref{eq:chrewrt} shows that
$\hat{\chi}_W(q)$ is a polynomial. Note also that, for an affine Coxeter group,
we only need to consider the elements $w_0(J)$ for $J\subsetneq S$ since
$w_0(S)$ does not exist, i.e.,
\begin{equation}\label{eq:chrewrt2}
  \hat{\chi}_W(q)=\sum_{J\subsetneq S} (-1)^{|J|}q^{\ell(w_0(J))}.  
\end{equation}

The Coxeter graph of an irreducible affine Coxeter group contains the Coxeter
graph of a finite Coxeter group as a subgraph. Using this observation we can
express the characteristic polynomial of an affine Coxeter group in terms of
those of finite Coxeter groups.

\begin{thm}\label{thm:infintofin}
 We have
\begin{align*}
  \hat{\chi}_{\tilde{A}_n}(q)&=\hat{\chi}_{A_n}(q)+\sum^n_{k=1} k(-1)^kq^{\binom{k+1}{2}}\hat{\chi}_{A_{n-k-1}}(q),\\
  \hat{\chi}_{\tilde{B}_n}(q)&=\hat{\chi}_{B_n}(q)+\sum_{k=0}^{n}(-1)^{k}q^{k(k-1)}\hat{\chi}_{B_{n-k}}(q),\\
  \hat{\chi}_{\tilde{C}_n}(q)&=\sum_{k=0}^{n}(-1)^kq^{k^2}\hat{\chi}_{B_{n-k}}(q),\\
  \hat{\chi}_{\tilde{D}_n}(q)&=\hat{\chi}_{D_n}(q)+(-1)^nq^{n(n-1)}+\sum_{k=0}^{n}(-1)^kq^{k(k-1)}\hat{\chi}_{D_{n-k}}(q),
\end{align*}
where we assume \( n\ge2 \), \( n\ge3 \), \( n\ge2 \), and \( n\ge 4 \) for \(
\tilde{A}_n \), \( \tilde{B}_n \), \( \tilde{C}_n \), and \( \tilde{D}_n \),
respectively. We define $\hat{\chi}_{A_i} (q) = \hat{\chi}_{B_i} (q)=1$ for $i\le 0$ and $\hat{\chi}_{D_i} (q)=1$ for $i\le 1$.
\end{thm}

\begin{proof}

  \begin{figure}
    \includegraphics[scale=0.8]{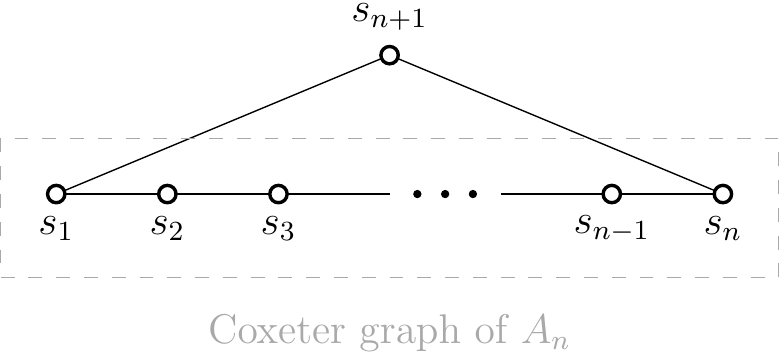}
    \caption{The Coxeter graph of \( \tilde{A}_n \) containing the Coxeter
      graph of \( A_n \).}
    \label{fig:tA}
  \end{figure}

  \textbf{Type $\tilde{A}_n$}: The Coxeter graph of $W=\tilde{A}_n$ contains
  that of $A_n$ as shown in Figure~\ref{fig:tA}. Let
  $S=\{s_1,s_2,\dots,s_{n+1}\}$ be the set of generators. By \eqref{eq:chrewrt2},
  we have
\begin{equation}\label{eq:chi(tA)}
  \hat{\chi}_{\tilde{A}_n}(q)
  =\sum_{J\subsetneq S, s_{n+1}\notin J} (-1)^{|J|} q^{\ell(w_0(J))}
  +\sum_{J\subsetneq S, s_{n+1}\in J} (-1)^{|J|} q^{\ell(w_0(J))}.
\end{equation}
Note that the first sum of the right-hand side of \eqref{eq:chi(tA)} is
given by
\begin{equation}\label{eq:chi(tA)1}
  \sum_{J\subsetneq S, s_{n+1}\notin J} (-1)^{|J|} q^{\ell(w_0(J))} = \hat{\chi}_{A_n}(q).
\end{equation}
Thus we only need to compute the second sum.

\begin{figure}
  \includegraphics[scale=0.8]{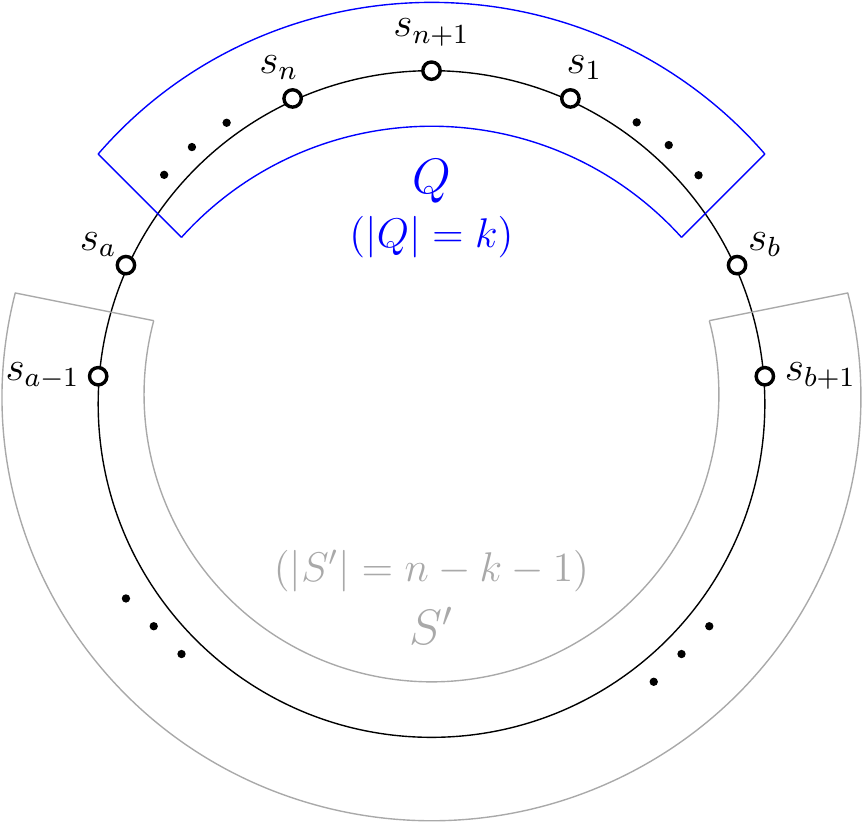}
  \caption{The Coxeter graph of \( \tilde{A}_n \) and the subset \( Q \).}
  \label{fig:tA2}
\end{figure}

Suppose \( J\subsetneq S \) with \( s_{n+1}\in J \). Let $Q\subseteq J$ be the
set of generators in the connected component containing $s_{n+1}$ and let
$|Q|=k$. We also define $J'=J\setminus Q$ so that $J=Q\cup J'$. If $k\ge n-1$,
then $J'=\emptyset$. Suppose $k< n-1$ and let $s_a,s_b$ be the generators that
are adjacent to an end vertex of $Q$ in the Coxeter graph of $\tilde{A}_n$, see
Figure~\ref{fig:tA2}. Then $J'$ is a subset of $S'=S\setminus (Q\cup
\{s_a,s_b\})$. Note that $W_{S'}\cong A_{n-k-1}$ since the graph of $S'$ is a
connected path with $|S'|=n-k-1$. Then
\[
  (-1)^{|J|} q^{\ell(w_0(J))}=(-1)^{|Q|+|J'|} q^{\ell(w_0(Q))+\ell(w_0(J'))}
  = (-1)^{k+|J'|} q^{\binom{k+1}{2}+\ell(w_0(J'))}.
\]
This implies that
\begin{equation}\label{eq:chi(tA)2}
  \sum_{s_{n+1}\in J\subsetneq S} (-1)^{|J|} q^{\ell(w_0(J))}
  =\sum^n_{k=1} k(-1)^kq^{\binom{k+1}{2}}\hat{\chi}_{A_{n-k-1}}(q).
\end{equation}
By \eqref{eq:chi(tA)}, \eqref{eq:chi(tA)1}, and \eqref{eq:chi(tA)2},
we obtain the formula for \( \hat{\chi}_{\tilde{A}_n}(q) \).

For the rest of the proof we denote by \(S=\{s_0,s_1,\dots,s_n\}\) the set of
generators of \(\tilde{B}_n\), \(\tilde{C}_n\), or \(\tilde{D}_n\).

\begin{figure}
  \includegraphics[scale=0.8]{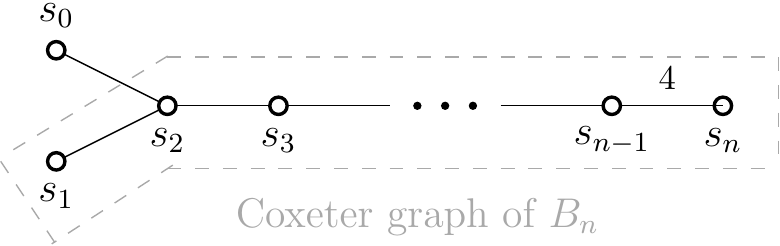}
  \caption{The Coxeter graph of \( \tilde{B}_n \) containing the Coxeter
    graph of \( B_n \).}
  \label{fig:tB}
\end{figure}

\textbf{Type $\tilde{B}_n$}: The Coxeter graph of $W=\tilde{B}_n$ includes the
graph of $B_n$ as shown in Figure~\ref{fig:tB}.
By \eqref{eq:chrewrt2}, we have
\begin{equation}
  \label{eq:chi(tB)}
  \hat{\chi}_{\tilde{B}_n}(q)
  = \sum_{J\subsetneq S,\{s_0,s_1\}\not\subseteq J} (-1)^{|J|}q^{\ell(w_0(J))}
  + \sum_{J\subsetneq S,\{s_0,s_1\}\subseteq J} (-1)^{|J|}q^{\ell(w_0(J))}.
\end{equation}
First we observe the case that \(\{s_0,s_1\}\not\subseteq J\).
If $s_0\notin J$, then since $J$ is a subset of the set of generators of $B_n$,
we have $\sum_{J\subsetneq S, s_0\notin J}
(-1)^{|J|}q^{\ell(w_0(J))}=\hat{\chi}_{B_n}(q)$. We have the same identity for
the case $s_1\notin J$. If $s_0, s_1\not\in J$, then we have
$\sum_{J\subsetneq S, s_0,s_1\notin J}
(-1)^{|J|}q^{\ell(w_0(J))}=\hat{\chi}_{B_{n-1}}(q)$. Thus,
\begin{equation}
  \label{eq:chiB2}
  \sum_{J\subsetneq S,\{s_0,s_1\}\not\subseteq J} (-1)^{|J|}q^{\ell(w_0(J))}=
  2\hat{\chi}_{B_n}(q)-\hat{\chi}_{B_{n-1}}(q).
\end{equation}

Suppose now that $\{s_0,s_1\}\subseteq J$. Let \( k \) be the smallest
integer such that \( s_k\not\in J \) and let \( Q=\{s_0,\dots,s_{k-1}\} \) and
$J'=J\setminus Q$. Then we have \( 2\le k\le n \), $W_Q\cong D_k$, and
\( J'\subseteq\{s_{k+1},\dots,s_n\} \). 
Since $\ell(w_0(Q))=k(k-1)$, we have
\[
  (-1)^{|J|}q^{\ell(w_0(J))} = (-1)^{|Q|+|J'|} q^{\ell(w_0(Q))+\ell(w_0(J'))}
  =(-1)^k q^{k(k-1)}(-1)^{|J'|} q^{\ell(w_0(J'))}.
\]
Therefore
\begin{align}
  \notag
  \sum_{J\subsetneq S,\{s_0,s_1\}\subseteq J} (-1)^{|J|}q^{\ell(w_0(J))}
  &= \sum_{k=2}^{n} (-1)^k q^{k(k-1)}\sum_{J'\subseteq\{s_{k+1},\dots,s_n\}}(-1)^{|J'|} q^{\ell(w_0(J'))}\\
  \label{eq:chiB3}
  &= \sum_{k=2}^{n} (-1)^k q^{k(k-1)}\hat{\chi}_{B_{n-k}}(q).
\end{align}

Combining  \eqref{eq:chi(tB)}, \eqref{eq:chiB2}, and \eqref{eq:chiB3} gives the
desired formula for  \( \hat{\chi}_{\tilde{B}_n}(q) \).

\begin{figure}
  \includegraphics[scale=0.8]{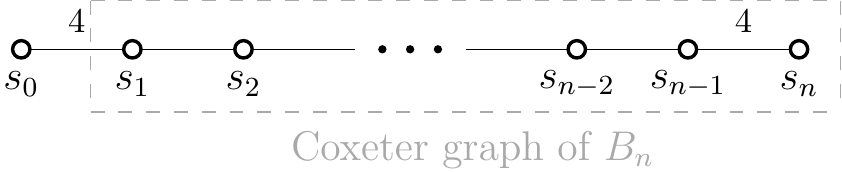}
  \caption{The Coxeter graph of \( \tilde{C}_n \) containing the Coxeter
    graph of \( B_n \).}
  \label{fig:tC}
\end{figure}

\textbf{Type $\tilde{C}_n$}: The Coxeter graph of $W=\tilde{C}_n$ includes the
graph of $B_n$ as shown in Figure~\ref{fig:tC}. By \eqref{eq:chrewrt2}, we have
\begin{equation}
  \label{eq:chi(tC)}
  \hat{\chi}_{\tilde{C}_n}(q)
  = \sum_{J\subsetneq S,s_0\not\in J} (-1)^{|J|}q^{\ell(w_0(J))}
  + \sum_{J\subsetneq S,s_0\in J} (-1)^{|J|}q^{\ell(w_0(J))}.
\end{equation}
This can be proved by similar arguments as in the proof of \(\tilde{B}_n\),
where in this case, for \( 1\le k\le n \), we have $W_Q\cong B_k$ and
$\ell(w_0(Q))=k^2$.

\begin{figure}
  \includegraphics[scale=0.8]{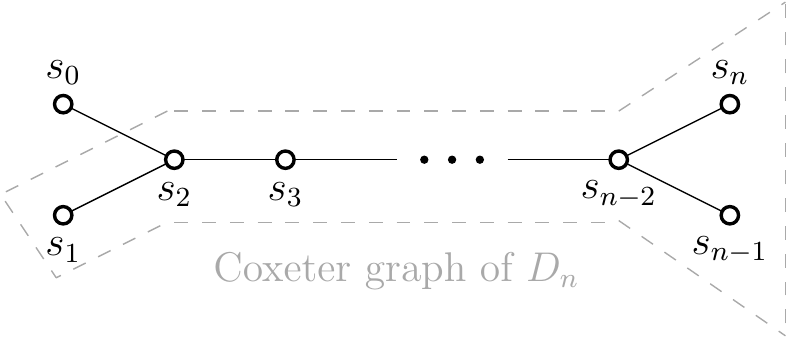}
  \caption{The Coxeter graph of \( \tilde{D}_n \) containing the Coxeter
    graph of \( D_n \).}
  \label{fig:tD}
\end{figure}

\textbf{Type $\tilde{D}_n$}: The Coxeter graph of $W=\tilde{D}_n$ includes the graph of $D_n$ as shown in Figure~\ref{fig:tD}.
By \eqref{eq:chrewrt2}, we have
\begin{equation}
  \label{eq:chi(tD)}
  \hat{\chi}_{\tilde{D}_n}(q)
  = \sum_{J\subsetneq S,\{s_0,s_1\}\not\subseteq J} (-1)^{|J|}q^{\ell(w_0(J))}
  + \sum_{J\subsetneq S,\{s_0,s_1\}\subseteq J} (-1)^{|J|}q^{\ell(w_0(J))}.
\end{equation}

 We can obtain the formula by similar arguments as in the proof of $\tilde{B}_n$. 
  If $\{s_0,s_1\}\nsubseteq J$, we have
\begin{equation}
  \label{eq:chiD2}
  \sum_{J\subsetneq S,\{s_0,s_1\}\not\subseteq J} (-1)^{|J|}q^{\ell(w_0(J))}=2\hat{\chi}_{D_n}(q)-\hat{\chi}_{D_{n-1}}(q).
\end{equation}

Suppose that $\{s_0,s_1\}\subseteq J$.
Let \( k \) be the smallest integer such that \( s_k\not\in J \) and let \( Q=\{s_0,\dots,s_{k-1}\} \) and $J'=J\setminus Q$. 
Then, for \( 2\le k\le n-2 \) or \(k=n\), we have $W_Q\cong D_k$ and
\( J'\subseteq\{s_{k+1},\dots,s_n\} \).
For \(k=n-1\), there are two cases such that \(Q\) and \(Q\cup\{s_n\}\).
Then we have \(W_Q\cong D_{n-1}\) and \(W_{Q\cup\{s_n\}}\cong D_n\).
In both cases we have \(J'=\emptyset\).
Therefore
\begin{align}
  \notag
  \sum_{J\subsetneq S,\{s_0,s_1\}\subseteq J} (-1)^{|J|}q^{\ell(w_0(J))}
  &= \sum_{k=2}^{n-2} (-1)^k q^{k(k-1)}\sum_{J'\subseteq\{s_{k+1},\dots,s_n\}}(-1)^{|J'|} q^{\ell(w_0(J'))}\\
  \notag
  &\qquad +(-1)^{n-1}q^{(n-1)(n-2)}+2(-1)^n q^{n(n-1)}\\
  \label{eq:chiD3}
  &= \sum_{k=2}^{n-2} (-1)^k q^{k(k-1)}\hat{\chi}_{D_{n-k}}(q)+(-1)^{n-1}q^{(n-1)(n-2)} +2(-1)^n q^{n(n-1)}.
\end{align}

Combining  \eqref{eq:chi(tD)}, \eqref{eq:chiD2}, and \eqref{eq:chiD3} gives the
desired formula for  \( \hat{\chi}_{\tilde{D}_n}(q) \).
\end{proof}

Note that using Proposition~\ref{prop:chmodch} the polynomials in
Theorem~\ref{thm:infintofin} can also be written as sums of characteristic
polynomials of finite Coxeter groups.

\section*{Acknowledgments}

The authors would like to thank Seung Jin Lee and Yibo Gao for helpful
discussions.

\end{document}